\documentclass{amsart}
\usepackage{amssymb, url, hyperref}
\usepackage[all]{xy}

\DeclareMathOperator{\GL}{GL}
\DeclareMathOperator{\PGL}{PGL}
\DeclareMathOperator{\SL}{SL}
\DeclareMathOperator{\sgn}{sgn}
\DeclareMathOperator{\Gal}{Gal}
\DeclareMathOperator{\Ind}{Ind}
\DeclareMathOperator{\St}{St}
\DeclareMathOperator{\N}{N}
\DeclareMathOperator{\tr}{Tr}
\DeclareMathOperator{\End}{End}
\DeclareMathOperator{\Aut}{Aut}
\DeclareMathOperator{\Sym}{Sym}
\DeclareMathOperator{\finite}{fin}

\DeclareMathOperator{\Sh}{Sh}

\newcommand{\F}{{\mathbf{F}}}
\newcommand{\bS}{{\mathbf{S}}}

\newcommand{\Z}{{\mathbf{Z}}}
\newcommand{\Q}{{\mathbf{Q}}}
\newcommand{\C}{{\mathbf{C}}}
\newcommand{\R}{{\mathbf{R}}}
\newcommand{\A}{{\mathbf{A}}}
\newcommand{\OO}{{\mathcal{O}}}
\newcommand{\gp}{{\mathfrak{p}}}

\newcommand{\eps}{\varepsilon}
\newcommand{\from}{{\colon}}
\newcommand{\isom}{\cong}
\newcommand{\injects}{\hookrightarrow}

\newcommand{\abs}[1]{\left\lvert #1 \right\rvert}
\newcommand{\set}[1]{\left\{ #1 \right\}}

\newcommand{\floor}[1]{\left\lfloor #1 \right\rfloor}
\newcommand{\tbt}[4]{\begin{pmatrix}#1 & #2\\#3 & #4\end{pmatrix}}
\newcommand{\ptbt}[4]{\left(\! \!\begin{pmatrix}#1 & #2\\#3 & #4\end{pmatrix}\!\!\right)}
\newcommand{\pnew}{{p\text{-new}}}
\newcommand{\pold}{{p\text{-old}}}

\numberwithin{equation}{subsection}

\newtheorem{theorem}{Theorem}[section]
\newtheorem{lemma}[theorem]{Lemma}
\newtheorem{cor}[theorem]{Corollary}
\newtheorem{prop}[theorem]{Proposition}
\newtheorem{problem}[theorem]{Problem}

\theoremstyle{definition}
\newtheorem{defn}[theorem]{Definition}

\newtheorem{rmk}[theorem]{Remark}
\newtheorem{example}{Example}

\title{On the computation of local components of a newform}

\subjclass[2010]{Primary: 11F70. Secondary: 11F11, 11Y99.}

\author{David Loeffler}
\address{Warwick Mathematics Institute \\ University of Warwick \\ Coventry CV4 7AL \\ UK}
\email{d.loeffler.01@cantab.net}
\thanks{The first author's research is supported by EPSRC Postdoctoral Fellowship EP/F04304X/2}

\author{Jared Weinstein}
\address{Mathematics Department \\ UCLA \\ Los Angeles, CA 90095-1555 \\ USA}
\email{jared@math.ucla.edu}
\thanks{The second author's research is supported by NSF Postdoctoral Fellowship DMS-0803089}

\begin{document}

\maketitle

\section{Introduction}

\subsection{The problem}

Let $f$ be a cuspidal newform for $\Gamma_1(N)$ with weight $k\geq 2$ and character $\varepsilon$. There are well-established methods for computing such forms using modular symbols, see~\cite{SteinModularForms}. Let $\pi_f$ be the corresponding automorphic representation of the ad{\`e}le group $\GL_2(\mathbf{A}_\Q)$. (As there are numerous possible normalisations, we shall recall the construction in section \ref{cuspforms} below.) This representation decomposes as a restricted tensor product over all places $v$ of $\Q$:
\[ \pi_f = \bigotimes_v\pi_{f,v}, \]
where $\pi_{f,v}$ is an irreducible admissible representation of $\GL_2(\Q_v)$. The isomorphism class of the infinite component $\pi_{f,\infty}$ is determined by the weight $k$ alone: it is the unique discrete series subrepresentation of the representation constructed via unitary induction from the character
\[\tbt{t_1}{*}{}{t_2}\mapsto \abs{\frac{t_1}{t_2}}^{k/2}\sgn(t_1)^k\]
of the Borel subgroup of $\GL_2(\R)$ \cite[prop.~5.21]{GelbartAutomorphicForms}. Similarly, the component $\pi_{f,p}$ for a prime $p\nmid N$ is the unramified principal series with Satake parameters equal to the roots of the polynomial $X^2 - a_p X + \eps(p) p^{k-1}$. We are led naturally to the following problem:

\begin{problem}\label{mainprob} Compute the local components $\pi_{f,p}$ for each prime $p\mid N$.
\end{problem}

Problem~\ref{mainprob} may be recast from an arithmetic perspective as a question about the local behavior of Galois representations. Suppose the Fourier coefficients of $f$ lie in a number field $E$. Let
\[\rho_{f,\lambda}\from\Gal(\overline{\Q}/\Q)\to\GL_2(E_\lambda)\]
be the family of $\lambda$-adic representations attached to $f$ by Deligne.

\begin{problem}\label{mainprob2} Compute the restriction of $\rho_{f,\lambda}$ to $\Gal(\overline{\Q}_p/\Q_p)$ for every $p\mid N$ and every prime $\lambda$ of $E$ not dividing $p$.
\end{problem}

Problems~\ref{mainprob} and~\ref{mainprob2} are equivalent. This is a consequence of the theorem of Carayol~\cite{CarayolSurLesRepresentations}, which we now summarize. For an irreducible admissible representation $\pi_p$ of $\GL_2(\Q_p)$, with coefficients in an arbitrary algebraically closed field of characteristic zero, let $\sigma(\pi_p)$ be the corresponding two-dimensional Frobenius-semisimple Weil-Deligne representation of $\Q_p$ under the Hecke correspondence. The Hecke correspondence is a modification of the local Langlands correspondence which has the convenient property of being invariant under automorphisms of the coefficient field; see~\cite{CarayolSurLesRepresentations}, \S 0.5. If the coefficient field is the $\lambda$-adic field $E_\lambda$, where $\lambda$ does not divide $p$, there is a further correspondence $\sigma(\pi_p)\mapsto\sigma^{\lambda}(\pi_p)$ between Frobenius-semisimple Weil-Deligne representations and continuous Frobenius-semisimple represtations $\Gal(\overline{\Q}_p/\Q_p)\to\GL_2(E_\lambda)$ due to Grothendieck; c.f.~\cite{TateNumberTheoreticBackground}.

Let $f$ be a cusp form as above. Carayol's theorem states that $\sigma^{\lambda}(\pi_{p,f})$ agrees with the restriction of $\rho_{f,\lambda}$ to $\Gal(\overline{\Q}_p/\Q_p)$. Therefore if one has access to $\pi_{f,p}$ as in Problem~\ref{mainprob} one could also compute $\rho_{f,\lambda}\vert_{\Gal(\overline{\Q}_p/\Q_p)}$, at least if one has an explicit construction of the Hecke correspondence.

Our approach is as follows. When $\pi_{f, p}$ is a principal series or special repesentation of $G = \GL_2(\Q_p)$, the problem is not difficult:
In these cases the isomorphism class of $\pi_{f,p}$ is determined by the $p$th Fourier coefficient of $f$ and those of its twists by an explicit set of Dirichlet characters of $p$-power conductor.  We are therefore reduced to the case of $\pi_{f,p}$ supercuspidal.  Here $\pi_{f, p}$ is of the form $\Ind_K^G \tau$, for a finite-dimensional representation $\tau$ of a maximal compact-modulo-centre subgroup $K \subseteq G$. Using results of Casselman on the structure of irreducible representations of $\GL_2(\Q_p)$, we find a canonical model of $\tau$ inside $\pi_{f, p}$, which can be explicitly realised as a subspace of the cohomology of a modular curve. We then compute this space, or rather its dual, using the explicit presentation of the homology of modular curves provided by modular symbols.

An alternative approach is provided by the local converse theorem of \cite[\S 27]{BushnellHenniartLocalLanglands}, which shows that an irreducible supercuspidal $G$-representation $\pi$ is determined by the local epsilon factors of its twists by a certain set of characters. While this approach is of enormous theoretical importance, it is less convenient from a computational perspective, as the description of the representation $\pi$ given by the epsilon factors of its twists is very inexplicit and hence difficult to relate to other more familiar parametrisations of the set of representations of $G$.

This article grew out of work done by the second author during his visit to the Magma Group in Sydney in 2009.  A partial version of the algorithm described here, in which $\pi_{f,p}$ is computed up to a quadratic twist, is implemented in the computational algebra system Magma~\cite{Magma}. The slightly more refined version presented here is the result of discussions between the authors at the workshop ``Computer methods for $L$-functions and automorphic forms'' hosted by the Centre de Recherches Math\'ematiques, Montr\'eal, in March 2010, and an implementation by the first author will be included in a forthcoming version of the Sage system \cite{Sage}.

\section{Basic Observations}

\subsection{Cusp forms and automorphic representations}
\label{cuspforms}

We begin by recalling how to construct automorphic representations from classical modular eigenforms, following \cite[\S 3.6]{Bump}. Let $G_\Q=\GL_2(\Q)$ and $G_{\A}=\GL_2(\A_\Q)$, and let $Z_\Q$ and $Z_{\A}$ be the centers of those groups. Let us fix a level $N$ and a Dirichlet character $\eps$ modulo $N$. Let $\omega$ be the adelization of $\eps$:  that is, $\omega$ is the unique character of $Z_\A / Z_\Q$ such that $\omega(\varpi_p) = \eps(p)$ for primes $p \nmid N$, where $\varpi_p\in \Q_p^\times\subset Z_\A$ is a uniformiser at $p$. Let $\omega_p$ be the restriction of $\omega$ to $\Q_p^\times$.

\begin{rmk} Note that the restriction of $\omega_p$ to $\Z_p^\times$ is the \textit{inverse} of the character of $\Z_p^\times$ inflated from the $p$-component of $\eps$.
\end{rmk}

We write
\[ K_0(N)=\set{\tbt{a}{b}{c}{d}\in \GL_2(\widehat \Z) \ \middle\vert\ c\equiv 0\pmod{N}},
\]
so that $K_0(N)$ is an open compact subgroup of $\GL_2(\A_\Q^{\finite})$. By \cite[Theorem 3.3.1]{Bump}, we have $G_\A = G_\Q G_\infty^+ K_0(N)$, where
\[G_\infty^+ = \set{g \in \GL_2(\R)\ \middle\vert\ \det(g) > 0}.\]
We define a character $\lambda$ of $K_0(N)$ by
\[ \lambda\ptbt a b c d = \prod_{p \mid N} \omega_p (d_p);\]
this clearly agrees with $\omega$ on $Z_\A \cap K_0(N)$.

Given an eigenform $f \in S_k(\Gamma_1(N),\eps)$, we define a function $\phi_f$ on $G_\A$ by
\[ \phi_f(\gamma g_\infty k_0) = f(g_\infty(i)) \cdot (ci+d)^{-k}\cdot (\det g_\infty)^{k / 2} \cdot \lambda(k_0),\]
for $\gamma \in G_\Q$, $g_\infty \in G_\infty^+$, and $k_0 \in K_0$. This is well-defined, since $G_\Q \cap G_\infty^+ K_0 = \Gamma_0(N)$ and $f$ satisfies $f(\frac{az + b}{cz + d}) = \eps(d)(cz + d)^k f(z)$ for $\tbt a b c d \in \Gamma_0(N)$; it is clear that $\phi_f(gz)=\omega(z)\phi_f(g)$ for $z\in Z_{\A}$.

The function $\phi_f$ in fact lies in the space $\mathcal{A}_0(\omega)$ of automorphic forms on $G_\A$ of central character $\omega$; see \cite[\S 3]{GelbartAutomorphicForms}. We define $\pi_f$ to be the linear span of the right translates of $f$ under $G_\A$;  then $\pi_f$ is an admissible smooth representation of $G_\A$. If $f$ is an eigenform, then it is well known that this representation is irreducible and decomposes as a restricted tensor product $\bigotimes'_v \pi_{f,v}$ of admissible irreducible representations of $\GL_2(\Q_v)$.

\subsection{New vectors}

Let us fix a prime $p$, and let $G=\GL_2(\Q_p)$. We recall a theorem of Casselman on smooth representations of $G$:

\begin{theorem}[{\cite{CasselmanOnSomeResults}}]
\label{newvector}
Let $\pi$ be an irreducible smooth representation of $G$ on a vector space $V$ (over a field of characteristic 0), with central character $\omega$. For large enough $r \geq 0$, there exists a nonzero vector $v \in V$ satisfying
\[\pi \ptbt a b c d v=\omega(d)v\]
for all $\tbt a b c d\in \GL_2(\Z_p)$ with $c\equiv 0\pmod{p^r}$. If $r$ is the minimal such exponent, then $v$ is unique up to scalar.
\end{theorem}

\begin{rmk}
 Our conventions differ slightly from those of \cite{CasselmanOnSomeResults}. Casselman uses $\omega(a)$ rather than $\omega(d)$, but his theorem and Thm.~\ref{newvector} are clearly equivalent.
\end{rmk}

\begin{defn} With the notation of Thm.~\ref{newvector}, the {\em conductor} of $\pi$ is $r=r(\pi)$, and a {\em new vector} for $\pi$ is $v$.
\end{defn}

One checks easily that the space of new vectors is preserved by the endomorphism
\[ \mathbf{T}_p =
\frac{1}{\sqrt{p}} \times \begin{cases}
\sum_{i = 0}^{p-1} \pi \ptbt p i 0 1 & (r \ge 1)\\
\sum_{i = 0}^{p-1} \pi \ptbt p i 0 1 + \omega(p) \cdot \pi \ptbt 1 0 0 p & (r = 0)
\end{cases}.\]
Since the space of new vectors is one-dimensional, we must have $\mathbf{T}_p v = \lambda v$ for some scalar $\lambda$. We define $\lambda(\pi)$ to be this scalar.

Smooth irreducible representations of $G$ (other than the trivial 1-dimensional representations factoring through the determinant) fall into three classes, which we summarise in Table~\ref{reptable} below (see for instance~\cite[\S 9.11]{BushnellHenniartLocalLanglands})

\begin{table}[ht]
\begin{tabular}{|l|cccc|}
\hline
Class & Notation & Conductor & Central char. & $\lambda(\pi)$ \\
\hline
Principal Series & $\pi(\chi_1,\chi_2)$ & $f_1+f_2$ & $\chi_1\chi_2$ & $\chi_1^*(p)+\chi_2^*(p)$\\
Special & $\St\otimes\chi_1$ & $\min(2f_1,1)$ & $\chi_1^2$ & $p^{-1/2}\chi_1^*(p)$ \\
Supercuspidal & $\Ind_K^G \tau$ & $\geq 2$ & $\omega_\tau$ & 0\\
\hline
\end{tabular}
\caption{The three classes of irreducible infinite-dimensional smooth representations of $G = \GL_2(\Q_p)$}
\label{reptable}
\end{table}

In the table, $\chi_1$ and $\chi_2$ are characters of $\Q_p^\times$ of conductors $p^{f_1}$ and $p^{f_2}$, and
\[
\chi_i^*(p)=\begin{cases}
\chi_i(p),&\text{$\chi_i$ is unramified} \\
0,&\text{$\chi_i$ is ramified}.
\end{cases}
\]

The representation denoted $\pi(\chi_1,\chi_2)$ is the principal series representation attached to the characters $\chi_1$ and $\chi_2$, defined whenever $\chi_1 / \chi_2 \ne |\cdot|^{\pm 1}$. The representation denoted $\St$ is the irreducible infinite-dimensional quotient of $\Ind_B^G 1$ (unnormalized induction), where $B\subset G$ is the subgroup of upper-triangular matrices.  The supercuspidal representations of $G$ are induced from certain finite-dimensional characters $\tau$ of compact-mod-center subgroups $K\subset G$;  these will be reviewed in \S\ref{Generalitiescuspidaltypes}. We have written $\omega_\tau$ for the central character of $\tau$.

If $f$ is a $p$-new eigenform, the automorphic form $\phi_f$ constructed above is a new vector for $\pi_{f, p}$ for every prime $p$.  One checks that $\lambda(\pi_{f, p}) = a_p(f) / p^{(k-1)/2}$. In particular, if $p \nmid N$ then $\pi_{f, p} \isom \pi(\chi_1, \chi_2)$ with $\chi_1$ and $\chi_2$ the unramified characters whose values at $p$ are $\{\alpha / p^{(k-1)/2}, \beta / p^{(k-1)/2}\}$, for $\alpha$ and $\beta$ the roots of $X^2 - a_p X +  \omega(p)p^{k-1}$.

\subsection{Twisting}

In this section we make an important reduction, which in particular will allow us to immediately dispose of the non-supercuspidal cases of Problem~\ref{mainprob}.

\begin{defn}
We say $\pi$ is \textit{primitive} if $r(\pi) \le r \left(\pi \otimes (\psi \circ \det)\right)$ for any character $\psi$ of $\Q_p^\times$.
\end{defn}

The principal series representation $\pi(\chi_1, \chi_2)$ is primitive if and only if (up to exchanging $\chi_1$ and $\chi_2$) the character $\chi_2$ is unramified, in which case the conductor of $\pi(\chi_1,\chi_2)$ equals the conductor of $\chi_1$. Similarly, $\St \otimes \chi$ is primitive if and only if $\chi$ is unramified, in which case it has conductor 1 and unramified central character.

The notions of twisting by a character and of primitivity have the following global analogues. For $f \in S_k(\Gamma_1(N), \eps)$ a new eigenform, and $\mu$ a Dirichlet character modulo $M$, the twist $f_\mu$ is the modular form (of level at most $N M^2$) whose Fourier coefficients are given by $a_n(f_\mu) = \mu(n) a_n(f)$. We write $f \otimes \mu$ for the unique {\em newform} whose $p$th Fourier coefficient is $\mu(p) a_p(f)$ for all but finitely many primes $p$;  this may have level smaller than $N M^2$. Then $f\otimes \mu$ is the newform corresponding to the automorphic representation $\pi_f \otimes \left(\psi\circ\det\right)$, where $\psi$ is the adelic character corresponding to $\mu$ as in \S\ref{cuspforms}.

Later we will need the following technical observation:

\begin{lemma}\label{cmtwist}
 Let $f$ be a newform of weight $k \ge 2$, and let $p$ be prime. If there exists a nontrivial character $\chi$ of $p$-power conductor such that $f \otimes \chi = f$, then $\chi$ is quadratic with $\chi(-1) = -1$.
\end{lemma}

\begin{proof} If $f$ is a newform of weight $k \ge 2$ and $f$ does not have CM, then the set of primes $\ell$ such that $a_\ell(f) = 0$ is well known to have density zero. If $f \otimes \chi = f$, then we must have $a_\ell(f) = 0$ for all $\ell$ such that $\chi(\ell) \ne 1$; so if $\chi$ is nontrivial, $f$ must be CM, and $\chi$ the unique quadratic character corresponding to the CM field of $f$. In particular, $\chi(-1) = -1$.
\end{proof}

\begin{defn}\cite[p.236]{AtkinLiPseudoEigenvalues} The newform $f$ is {\em $p$-primitive} if the $p$-part of its level is minimal among all its twists.
\end{defn}

Clearly, $f$ is $p$-primitive if and only if $\pi_{f, p}$ is a primitive representation of $G$. If $f$ is $p$-primitive, then its $p$th Fourier coefficient, its central character and its weight already determine $\pi_{f,p}$ completely, {\em{unless}} $\pi_{f,p}$ is supercuspidal:

\begin{prop}
\label{supercuspidalcriterion}
 Let $\pi=\pi_{f,p}$ be the local component at $p$ of a $p$-primitive newform $f\in S_k(\Gamma_1(Np^r), \eps)$, with $p \nmid N$ and $r \ge 1$. Let $\omega_p$ be the $p$-component of the character of $\A^\times_\Q / \Q^\times$ corresponding to $\eps$.
\begin{enumerate}
\item If $r \ge 1$ and the conductor of $\omega_p$ is $p^r$, then $\pi\isom\pi(\chi_1,\chi_2)$, where $\chi_1$ is unramified, $\chi_1(p)=a_p(f)/p^{(k-1)/2}$, and $\chi_2$ is determined by $\chi_1\chi_2=\omega_p$.
\item If $r=1$ and $\omega_p$ is unramified, then $\pi_p\isom\St\otimes\chi$, where $\chi$ is the unramified character with $\chi(p)=a_p(f)/p^{(k-2)/2}$.
\item If neither of the above conditions hold, then $\pi$ is supercuspidal and the conductor of $\omega_p$ is at most $\floor{r/2}$.
\end{enumerate}
\end{prop}

\begin{proof}
The only statement requiring proof is that a primitive supercuspidal representation of conductor $r$ must have central character of conductor $\le \floor{r/2}$; this is \cite[Theorem 4.3']{AtkinLiPseudoEigenvalues}. The remaining statements are clear from Table \ref{reptable} and the relation $\lambda(\pi_{f,p})=a_p(f)/p^{(k-1)/2}$.
\end{proof}

\begin{cor}
Suppose $\pi_{f, p}$ is not supercuspidal. Let $\chi$ be any character such that $f \otimes \chi$ is $p$-primitive. Then $\pi_{f,p}$ is determined up to isomorphism by $\eps$, $\chi$, $k$, and $a_p(f \otimes \chi)$.
\end{cor}

Hence it is the supercuspidal cases that are most interesting.

\section{Supercuspidal Representations}

\subsection{Generalities on cuspidal types}
\label{Generalitiescuspidaltypes}

Let $\pi$ be a supercuspidal representation of $G=\GL_2(\Q_p)$. Then, by the induction theorem of \cite[\S 15.5]{BushnellHenniartLocalLanglands}, $\pi$ arises by induction from a finite-dimensional character of an (explicit) open subgroup $K\subset G$ which contains the center $Z$ of $G$ and for which $K/Z$ is compact.  The following definition is not standard but is well-suited for our purposes:

\begin{defn} A {\em cuspidal type} is a pair $(K,\tau)$, where $K$ is a maximal compact-mod-center subgroup of $G$ and $\tau$ is a finite-dimensional character of $K$ such that $\Ind_K^G\tau$ is an irreducible supercuspidal representation of $G$.
\end{defn}

Up to conjugacy, there are two maximal compact-mod-center subgroups of $G$: one is $\Q_p^\times\GL_2(\Q_p)$, and the other is the normalizer in $G$ of an Iwahori subgroup $I_0(p)=\tbt{\Z_p^\times}{\Z_p}{p\Z_p}{\Z_p^\times}\subset \GL_2(\Z_p)$. We will refer to these as the {\em unramified} and {\em ramified} maximal compact-mod-center subgroups, respectively.

In either case, $K$ comes along with a natural filtration by normal subgroups. Let $K_0\subset K$ be the maximal compact subgroup, so that $K_0$ equals $\GL_2(\Z_p)$ or $I_0(p)$. Then $K_0$ is the group of units in a $\Z_p$-algebra $A\subset M_2(\Z_p)$ equal to $M_2(\Z_p)$ in the unramified case and $\tbt{\Z_p}{\Z_p}{p\Z_p}{\Z_p}$ in the ramified case. The Jacobson radical $P$ of $A$ is $pA$ or $\tbt{p\Z_p}{\Z_p}{p\Z_p}{p\Z_p}$. Let $K_n=1+P^n$; then the $K_n$ are normal in $K$.

\begin{defn}
Let $(K,\tau)$ be a cuspidal type. The {\em level} of $(K,\tau)$ is the least integer $n\geq 0$ for which $\tau$ factors through $K/K_n$. The type $(K,\tau)$ is {\em primitive} if $\pi=\Ind_K^G\tau$ is a primitive representation of $G$.
\end{defn}

\begin{theorem} \label{exhaustion} Every irreducible supercuspidal representation of $G$ arises from a cuspidal type $(K,\tau)$. If $K$ is ramified, then one may take the level of $\tau$ to be even.
\end{theorem}

\begin{proof}
For a proof that every irreducible supercuspidal representation arises from a cuspidal type, see~\cite{BushnellHenniartLocalLanglands}, \S15.8. There an exhaustive list of cuspidal types $(K,\tau)$ is given;  these cuspidal types are classified according to which ``simple strata'' they contain. For the definition of simple stratum, see~\cite{BushnellHenniartLocalLanglands}, \S13.1 and \S13.2. In the definition of ramified simple stratum $(A,n,\alpha)$ appearing in \S13.1, the integer $n$ is required to be odd;  in our notation, the level of any $\tau$ containing such a stratum has level $n+1$, which is therefore even.
\end{proof}

We remark that in~\cite{BCDT}, a concise list of pairs $(K,\tau)$ is given for which $\Ind_K^G\tau$ exhausts the set of isomorphism classes of supercuspidal representations, at least in the case of $p$ odd.

\begin{prop}
\label{nversusr} Let $(K,\tau)$ be a primitive cuspidal type of level $n$. Let $r$ be the conductor of $\Ind_K^G\tau$.  Then $r=2n$ if $K$ is unramified and $r=n+1$ if $K$ is ramified.
\end{prop}

\begin{proof}   The relationship between the level of $(K,\tau)$ and the conductor of $\Ind_K^G\tau$ is spelled out in Henniart's appendix to~\cite{BreuilMezard}, A3.2--A.3.8.
\end{proof}

\subsection{The restriction of a cuspidal representation to a compact-mod-center subgroup}

Suppose $f$ is a $p$-primitive cusp form for which one has determined (using Prop.~\ref{supercuspidalcriterion}, say) that the local component $\pi_{f,p}$ is supercuspidal. By Thm.~\ref{exhaustion}, there exists a cuspidal type $(K,\tau)$ for which $\pi_{f,p}=\Ind_K^G\tau$. The subgroup $K$ is determined (up to conjugacy in $\GL_2(\Z_p)$) by the parity of the exponent of $p$ in the conductor of $f$. Therefore to complete the solution to Problem.~\ref{mainprob} it suffices to compute a model for the $K$-module $\tau$. We shall see how to compute $\tau$ in Thm.~\ref{identifytypespace}. In preparation for this we shall need to gather some purely local results on supercuspidal representations of $\GL_2(\Q_p)$.

We will have use for an old result of Casselman, see Theorem 3 of~\cite{CasselmanTheRestriction} and its proof. For the moment, let $\Gamma(p^m)\subset\GL_2(\Z_p)$ denote the group of matrices congruent to the identity modulo $p^m$ (this agrees with $K_m$ if $K$ is unramified). For a ring $R$, let $N(R)\subset\GL_2(R)$ be the unipotent subgroup $\tbt{1}{R}{}{1}$.

\begin{theorem}
\label{casselman} Let $\pi$ be a primitive cuspidal representation of conductor $p^r$, and let $m=\floor{(r+1)/2}$. Then $\pi^{\Gamma(p^m)}$ is an irreducible representation of $\GL_2(\Z/p^m\Z)$.
Furthermore, the restriction of $\pi^{\Gamma(p^m)}$ to $N(\Z/p^m\Z)\isom \Z/p^m\Z$ is determined as follows:
\begin{enumerate}
\item If $r$ is even, then $\pi^{\Gamma(p^m)}\vert_{N(\Z/p^m\Z)}$ is the direct sum of characters of conductor $p^m$. Each of these $p^{m-1}(p-1)$ characters appears with multiplicity one.
\item If $r$ is odd, then $\pi^{\Gamma(p^m)}\vert_{N(\Z/p^m\Z)}$ is the direct sum of characters of conductors $p^{m-1}$ and $p^m$. Each of these $p^{m-2}(p^2-1)$ characters appears with multiplicity one.
\end{enumerate}
\end{theorem}

Let $(K,\tau)$ be a primitive cuspidal type, and let $\pi=\Ind_K^G\tau$. We will use Casselman's theorem to construct a model of $\tau$ inside of $\pi$.

\begin{theorem}
\label{Kmoduleisomorphism}
Let $p^r$ be the conductor of $\pi$, let $v$ be a new vector for $\pi$, and let $w$ be the vector
\[ w=\pi\ptbt{p^{\floor{r/2}}}{0}{0}{1} v. \]
Finally, let $W$ be the span of the $N(\Z_p)$-translates of $w$.
\begin{enumerate}
\item We have $W=\pi^{K_n}$. In particular, since $K_n\subset K$ is normal, we have that $W$ is stabilized by the action of $K$.
\item The action of $K$ on $W$ is irreducible.
\item There is an isomorphism of $K$-modules $W\approx \tau$.
\end{enumerate}
\end{theorem}

\begin{proof}
First observe that (1) and (2) implies (3). Indeed, since $\pi$ is induced from $\tau$, and $\tau$ is fixed by $K_n$, we see that $\pi^{K_n}$ contains a $K$-module isomorphic to $\tau$. But by (1) and (2), $\pi^{K_n}=W$ is already an irreducible $K$-module, which forces $W\approx\tau$. Therefore it suffices to prove (1) and (2).

First suppose $r=2n$ is even, so that $K=\Q_p^\times{\GL_2(\Z_p)}$ is unramified. Then $K_n=\Gamma(p^n)$. The vector $w$ has the property that
\[ \pi\ptbt{a}{p^nb}{p^nc}{d} w = \omega(a)w, \]
where $a,d\in\Z_p^\times$, $b,c\in \Z_p$, and where $\omega$ is the central character of $\pi$. It immediately follows that $W\subset \pi^{K_n}$. Let $B(\Z_p)\subset \GL_2(\Z_p)$ be the subgroup of upper-triangular matrices. Since $B(\Z_p)$ preserves $\C w$ and normalizes $N(\Z_p)$ we see that $B(\Z_p)$ preserves $W$. Consider the decomposition of $W$ into irreducible characters of $N(\Z_p)$:  certainly each character appearing has conductor $\leq p^n$. On the other hand, by Thm.~\ref{casselman}, part (1), $\pi^{K_n}\vert_{N(\Z_p)}$ contains exactly those characters of conductor $p^n$. Therefore $W$ only contains characters of conductor $p^n$. But $B(\Z_p)$ preserves $W$:  as a result, {\em every} character of conductor $p^n$ occurs in $W$. In view of Thm.~\ref{casselman}, we have shown that $W$ is all of $\pi^{K_n}$. Furthermore, the action of $B(\Z_p)$ on $W$ is irreducible (since it acts transitively on the set of characters of $N(\Z_p)$ inside $W$). Hence the action of $K$ on $W$ is also irreducible, as required.

Now suppose that $K$ is ramified, so that $r$ is odd:  Write $r=2m-1$ and $n=2m-2$. We have that
\[ \pi\ptbt{a}{p^{m-1}b}{p^mc}{d} w = \omega(a)w, \]
for all $a,d\in\Z_p^\times$, $b,c\in\Z_p$. The central character $\omega$ has conductor $\leq p^{m-1}$. The group $K_n$ is
\[ K_n = \tbt{1+p^{m-1}\Z_p}{p^{m-1}\Z_p}{p^{m}\Z_p}{1+p^{m-1}\Z_p}\supset \Gamma(p^m),\]
from which we see that $W\subset \pi^{K_n}\subset \pi^{\Gamma(p^m)}$. Applying Thm.~\ref{casselman} shows that $\pi^{\Gamma(p^m)}$ is an irreducible representation of $\GL_2(\Z/p^m\Z)$ containing exactly those characters of $N(\Z_p)$ of conductors $p^{m-1}$ and $p^{m}$. Consider the restriction of $W$ to $N(\Z_p)$:  each character that appears must have conductor $p^{m-1}$. Since $W$ is preserved by $B(\Z_p)$, we see that $W\vert_{N(\Z_p)}$ is the sum of exactly those characters of conductor $p^{m-1}$. Since no characters of conductor $p^m$ could occur in $\pi^{K_n}\vert_{N(\Z_p)}$, we have that $W=\pi^{K_n}$. The action of $K$ on $W$ is again irreducible, for identical reasoning as in the previous paragraph.
\end{proof}

Let $U_p\subset \GL_2(\Z_p)$ be the subgroup of matrices $\tbt{a}{b}{c}{d}$ with $c\equiv 0\pmod{p^r}$ and $a,d\equiv 1\pmod{p^{\floor{r/2}}}$.  Observe that
\begin{equation}
\label{UpKn}
U_p =  \tbt{p^{\floor{r/2}}}{}{}{1}^{-1} K_n \tbt{p^{\floor{r/2}}}{}{}{1}.
\end{equation}
Therefore if $\pi$ is a primitive supercuspidal representation of conductor $r$, then $\pi^{U_p}$ is stabilized by the group
\begin{equation}
\label{Kprime}
K' = \tbt{p^{\floor{r/2}}}{}{}{1}^{-1} K \tbt{p^{\floor{r/2}}}{}{}{1}.
\end{equation}
Note that $K'$ contains the subgroup of matrices $\tbt{a}{b}{c}{d}\in\GL_2(\Z_p)$ with $c\equiv 0\pmod{p^r}$.
We will need the following corollary to Thm.~\ref{Kmoduleisomorphism}.

\begin{cor} \label{Kmodulecor} Let $\rho\from K'\to \Aut V$ be a continuous linear representation of $K'$ which is isomorphic to a direct sum of finitely many copies of $\pi^{U_p}$.  Let $v\in V$ be a nonzero vector with the property that
\[\rho\ptbt{a}{b}{c}{d}v = \omega(d)v\]
for all $\tbt{a}{b}{c}{d}\in\GL_2(\Z_p)$ with $c\equiv 0\pmod{p^r}$.
Let $X$ be the span of the translates of $v$ under $\tbt{1}{p^{-\floor{r/2}}\Z_p}{}{1}$.  Then $X$ is preserved by $K'$, and $(K',X)$ is a cuspidal pair with $\Ind_{K'}^G X \isom \pi$.
\end{cor}

\begin{proof}  Write $\rho=(\pi^{U_p})^{\oplus\mu}$ and suppose $v=(v_i)_{1\leq i\leq \mu}$, where $v_i$ is a vector in the space of $\pi^{U_p}$.  Then by the hypothesis on $v$ we see that $v_i$ is a new vector for $\pi$.  By uniqueness of the new vector up to scaling we must have $v=(a_iv_0)_{1\leq i\leq \mu}$ for a particular new vector $v_0$ and a nonzero tuple $(a_i)\in\C^\mu$.  Thus $v$ is contained in $\set{(a_ix)\vert x\in \pi^{U_p}}\subset V$ which is itself a model for $\pi^{U_p}$.  The result now follows from Thm.\ref{Kmoduleisomorphism}.
\end{proof}

\subsection{Admissible pairs}

If $p$ is odd, then there is an explicit parametrisation of the primitive supercuspidal types.

\begin{defn}[{\cite[\S 18]{BushnellHenniartLocalLanglands}}]
\label{admissiblepair}
 An \textit{admissible pair} is a pair $(E, \theta)$, where $E / \Q_p$ is a tamely ramified extension and $\theta$ is a character $E^\times \to \overline{\Q}^\times$.  These must satisfy the conditions:
\begin{enumerate}
\item $\chi$ does not factor through the norm map $\N_{E/\Q_p}\from E^\times\to \Q_p^\times$, and
\item If $E/\Q_p$ is ramified, then the restriction $\chi\vert U^1_E$ must not factor through $\N_{E/\Q_p}$.
\end{enumerate}
We say that $(E, \theta)$ and $(E', \theta')$ are equivalent if there is an isomorphism $\iota: E \to E'$ such that $\theta' \circ \iota = \theta$.
\end{defn}

In \cite[\S 20]{BushnellHenniartLocalLanglands}, a construction is given which associates to each admissible pair $(E, \theta)$ a cuspidal type $(K, \tau)$, where $K$ is ramified or unramified as $E$ is. If $\chi$ is a character of $\Q_p^\times$, then $(E, \theta \times (\chi \circ \N_{E/\Q_p}))$ corresponds to $(K, \tau \otimes (\chi \circ \det))$; hence $(K, \tau)$ is primitive if and only if for all characters $\chi$ of $\Q_p^\times$, we have $c(\theta\times \chi\circ\N_{E/\Q_p})\geq c(\theta)$.

For $p \ne 2$ this construction gives a bijection between admissible pairs (up to equivalence) and cuspidal types (up to isomorphism); if $p = 2$ then the admissible pairs correspond to a subset of the cuspidal types. We will need the following formula describing the trace of the representation $(K, \tau)$:

\begin{defn} An element $\alpha\in E^\times$ is minimal if:
\begin{enumerate}
\item $E/\Q_p$ is unramified and $\alpha=p^k\alpha_0$ for a unit $\alpha_0\in \OO_E^\times$ whose image in $\OO_E/\gp_E\isom\mathbf{F}_{p^2}$ does not lie in $\mathbf{F}_p$.
\item $E/\Q_p$ is ramified and $\alpha$ has odd valuation.
\end{enumerate}
\end{defn}

\begin{theorem} \label{charcalc} Let $(K,\tau)$ be a primitive cuspidal type of level $n$ corresponding to an admissible pair $(E, \theta)$, and let $\alpha\in E^\times$ be a minimal element.  Let $s\in\Gal(E/\Q_p)$ be the nontrivial automorphism.  Then
\[ \tr \tau(\alpha)=\iota\left(\theta(\alpha)+\theta(\alpha^s)\right),\]
where $\iota=(-1)^{(n-1)}$ if $E/\Q_p$ is unramified and $\iota=1$ otherwise.
\end{theorem}

\begin{proof} See \cite{WeinsteinHMFs}, proof of Lemma 2.2. \end{proof}

\begin{rmk}  Suppose $(E,\theta)$ is an admissible pair, and let $\pi_\theta=\Ind_K^G\tau$ be the corresponding supercuspidal representation of $\GL_2(\Q_p)$.  By local class field theory one can view $\theta$ as a character of the Weil group $W_E$.  Let $\sigma_\theta=\Ind_{E/\Q_p} \theta$;  the conditions on $\theta$ in definition~\ref{admissiblepair} ensure that $\sigma$ is an irreducible representation of $W_{\Q_p}$.

It is {\em{not}} the case that $\pi_\theta$ and $\sigma_\theta$ correspond under the local Langlands correspondence.  Rather, there is a character $\Delta_\theta$ of $E^\times$ of order dividing 4 for which $\pi_{\theta\Delta_\theta}$ and $\sigma_\theta$ correspond.  If $E/\Q_p$ is unramified then $\Delta_\theta$ is the unramified quadratic character.  If $E/\Q_p$ is ramified then $\Delta_\theta$ is a normalized Gauss sum, c.f.~\cite{BushnellHenniartLocalLanglands}, \S 34.4.
\end{rmk}

\section{Modular symbols}

\subsection{Definitions}
\label{modsymbdefinitions}

For $k \ge 2 \in \Z$ and $\Gamma$ a finite-index subgroup of $\SL_2(\Z)$, we let $\bS_k(\Gamma, \Q)$ denote the space of cuspidal modular symbols of level $\Gamma$ over $\Q$.  For the definition of this space, see \cite[\S 8.1]{SteinModularForms}. For our purposes, it suffices to state that:
\begin{enumerate}
 \item $\bS_k(\Gamma, \Q)$ is a finite-dimensional $\Q$-vector space, equipped with an action of the Hecke algebra $\mathcal{H}=\mathcal{H}(\GL_2^+(\Q) // \Gamma)$;
 \item Given an algorithm to determine whether or not a given $g \in \SL_2(\Z)$ lies in $\Gamma$, there are algorithms that will calculate a basis for $\bS_k(\Gamma, \Q)$ and the matrix of the Hecke operator $[\Gamma x \Gamma]$ on $\bS_k(\Gamma,\Q)$ for any given $x \in \GL_2^+(\Q)$;
 \item There is a perfect pairing between $\bS_k(\Gamma) = \bS_k(\Gamma, \Q) \otimes_\Q \C$ and $S_k(\Gamma) \oplus \overline{S_k(\Gamma)}$, with respect to which the action of $\mathcal{H}$ is self-adjoint;
 \item If a subgroup $\Gamma'\subset\SL_2(\Q)$ normalizes $\Gamma$, there is an action $\Gamma'\to\GL(\bS_k(\Gamma,\Q))$;  this action is denoted $(\gamma,\sigma)\mapsto\gamma\circ\sigma$.
\end{enumerate}

The subgroups $\Gamma$ we shall consider will all be of the form
\[\Gamma_H(N) = \left\{\tbt a b c d\ \middle|\ c = 0 \pmod{N}, d \in H\right\},\]
where $H$ is some subgroup of $(\Z / N\Z)^\times$. A highly optimised implementation\footnote{This implementation is the work of many authors, including William Stein, Jordi Quer, Craig Citro and the first author of this paper.}  of modular symbol algorithms for these level groups is included in the software package \cite{Sage}.

If $f$ is a new eigenform of level $N$, and the character $\eps$ of $f$ is trivial on the subgroup $H$, then $f \in S_k(\Gamma_H(N))$, and hence there is a two-dimensional eigenspace $M_f\subset\bS_k(\Gamma_H(N))$ with the same eigenvalues as $f$ for the Hecke operators $T_m$. This space is preserved by the Hecke operator corresponding to $\tbt 1 0 0 {-1}$, which has eigenvalues $\pm 1$ on $M_f$. Thus there are corresponding basis vectors $\sigma^+_f$, $\sigma^-_f$ of $M_f$, each unique up to scaling. We refer to these as the \textit{plus eigensymbol} and \textit{minus eigensymbol} for $f$.

\subsection{Determination of the minimal twist}

Let us now suppose $f$ is a new eigenform of level $Np^r$, with $p \nmid N$, and character $\eps$. Write $\eps = \eps_N \eps_p$ as a product of characters of $(\Z / N\Z)^\times$ and $(\Z / p^r \Z)^\times$.

\begin{prop}\label{twisting} Let the conductor of $\eps_p$ be $p^c$, and let $u = \min(\lfloor r/2 \rfloor, r - c)$.
\begin{enumerate}
\item If $\chi$ is a Dirichlet character of conductor $p^v$ with $v > u$, then $f_\chi$ is new of level $LCM(Np^r, Np^{c + v}, Np^{2v}) > Np^r$.
\item If $\chi$ is a character modulo $p^u$, then $f_\chi=f\otimes\chi$ has level $Np^r$ and is new at the primes $\ell \mid N$. Moreover, the span of the twists $f_\chi$ is equal to the span of the translates of $f$ by $\tbt 1 {p^{-u}\Z} 0 1$.
\end{enumerate}
\end{prop}

\begin{proof} The first part follows from results in \cite[\S 3]{AtkinLiPseudoEigenvalues}.

For the second part, let $X_f$ be the span of the translates of $f$ by $\tbt 1 {p^{-u}\Z} 0 1$.  It is immediate that the span of the twists of $f$ is contained in $X_f$. For the reverse inclusion: if $c = r$, or if $r = 1$ and $c = 0$, then $u = 0$ so there is nothing to prove; so we may assume that this is not the case. Then since $f$ is a new eigenform, we must have $a_n(f) = 0$ whenever $p \mid n$.

Hence a basis for $X_f$ is given by the series $\sum_{n = i \bmod p^u} a_n q^n$ where $i$ runs through  $(\Z / p^u \Z)^\times$, and each such series is clearly a linear combination of twists of $f$ by Dirichlet characters modulo $p^u$.  This proves the reverse inclusion.
\end{proof}

We now construct a space of modular symbols which ``sees'' all of the relevant twists of $f$. Let $H_N$ be any subgroup of $(\Z / N \Z)^\times$ contained in the kernel of $\eps_N$, and let $H_p$ be the subgroup of $(\Z / p^r \Z)^\times$ which is the kernel of reduction modulo $p^{\max(c, \floor{r/2})}$. We let $H = H_N \times H_p \subseteq (\Z / Np^r\Z)^\times$. Then $\tbt 1 {p^{-u}} 0 1$ normalises $\Gamma = \Gamma_H(Np^r)$.

\begin{defn}
\label{typespace}
The {\em type space} $X^+_f$ is the subspace of $\bS_k(\Gamma)$ spanned by the translates of the plus eigensymbol $\sigma^+_f$ under $\tbt 1 {p^{-u}\Z_p} 0 1$.
\end{defn}

If $g$ is a newform in $S_k(\Gamma_1(M)$ for some $M \mid Np^r$, whose character is trivial on $H$, we write $\bS_k(\Gamma)[g]$ for the subspace of $\bS_k(\Gamma)$ on which the Hecke operators $T_m$ for $(m,Np)=1$ act as multiplication by $a_m(g)$.  Just as with spaces of modular forms, we have a decomposition
\[\bS_k(\Gamma) = \bS_k(\Gamma)_\pold \oplus \bS_k(\Gamma)_\pnew\]
where the summands are respectively spanned by the spaces $\bS_k(\Gamma)[g]$ for $g$ of $p$-level $< p^r$ and $p$-level exactly $p^r$ respectively;  these two spaces can be explicitly computed using degeneracy maps, as in \cite[\S 8.6]{SteinModularForms}.

The following is a consequence of Prop.~\ref{twisting}.

\begin{prop} Let $f\in S_k(\Gamma)$ be a newform of level $Np^r$, and let $g\in S_k(\Gamma)$ be a newform whose level divides $Np^r$.
The following are equivalent:
\begin{enumerate}
\item There exists a Dirichlet character $\chi$ of $p$-power conductor with $g=f\otimes\chi$.
\item $X^+_f \cap \bS_k(\Gamma)[g] \ne 0$.
\end{enumerate}
In particular, $f$ is $p$-primitive if and only if $X^+_f \subseteq \bS_k(\Gamma)_\pnew$.
\end{prop}

Hence the type space $X_f^+$ allows us to determine algorithmically whether or not $f$ is $p$-primitive.

\subsection{The type space and cuspidal pairs.}
\label{typespacecuspidalpairs}
When $f$ is $p$-primitive and $\pi_{f,p}$ is supercuspidal, then in fact $X_f^+$ tells us considerably more about $\pi_{f, p}$, as we shall show. As we have seen from Prop.~\ref{supercuspidalcriterion}, if $\pi_{f, p}$ is primitive and supercuspidal then we have $c \le \floor{r/2}$, so $H_p$ is the subgroup of elements of $\Z / p^r\Z$ which are 1 modulo $p^{\floor{r/2}}$.

As in \S~\ref{Generalitiescuspidaltypes}, let $K$ denote either the unramified or the ramified maximal compact-modulo-centre subgroup of $G = \GL_2(\Q_p)$, for $r$ even or odd respectively, and let $S(K)$ denote the intersection of $K$ with $\SL_2(\Q_p)$.  We observe that $K$ is generated by $S(K)$ along with the following subgroups and elements:
\begin{itemize}
\item The torus $T=\tbt{\hat{\Z}_p^\times}{}{}{\hat{\Z}_p^\times}$
\item The matrix $pI_2=\tbt{p}{}{}{p}$, if $K$ is unramified
\item The matrix $\Pi=\tbt{}{1}{-p}{}$, if $K$ is ramified.
\end{itemize}
Let $\tilde{S}(K)$ denote the subgroup of $\GL_2(\Q_p)$ generated by $S(K)$ and the single matrix $pI_2$ or $\Pi$ as $K$ is unramified or ramified, respectively.  Then $\tilde{S}(K)$ normalises the subgroup $S(K_n) = K_n \cap \SL_2(\Q_p)$.

The closure of $\Gamma$ in $\SL_2(\Q_p)$ is $\beta^{-1} S(K_n) \beta$, where $\beta=\tbt {p^{\floor{r/2}}} {}{} 1$ is the matrix from Thm.~\ref{Kmoduleisomorphism} and $n$ is the integer determined by Prop.~\ref{nversusr}.  As $S(K_n)$ is normal in $\tilde{S}(K)$, this allows us to define an action $\rho\from \tilde{S}(K)\to \GL(\bS_k(\Gamma))$. This action may be described concretely as follows.  For $\alpha\in \tilde{S}(K)$ and $\sigma \in \bS_k(\Gamma)$, let $\tilde{\alpha}=\tbt{a}{b}{c}{d}\in\GL_2(\Q)$ be a matrix with integer entries satisfying
\begin{enumerate}
\item The image of $\tilde{\alpha}$ in $\GL_2(\Q_p)$ lies in $\alpha S(K_n)$,
\item $c\equiv 0\pmod{N}$, and $d\bmod N\in H$, and
\item $\det \tilde{\alpha}$ is a power of $p$.
\end{enumerate}

Then the matrix $\beta^{-1}\tilde{\alpha}\beta$ normalizes $\Gamma$, and we may define
\[ \rho(\alpha)\sigma = \beta^{-1} \tilde{\alpha} \beta \circ \sigma. \]
Note that in the case of $K$ ramified, the operator $\rho(\Pi)$ is $\eps_N(p^{\floor{r/2}}) \cdot W_{p^r}$, where $W_{p^r}$ is the Atkin-Lehner operator; this follows from \cite[proposition 1.1]{AtkinLiPseudoEigenvalues}.

The next theorem plays a crucial role in the determination of the local component of a newform when that component is supercuspidal.

\begin{theorem}
\label{identifytypespace}
Let $f$ be a $p$-primitive newform of weight $k$, level $N p^r$ and nebentypus $\eps$. Assume that $\pi_{f,p}$ is a supercuspidal representation of $G_p=\GL_2(\Q_p)$. Let $X_f^+\subset \bS_k(\Gamma)$ be the subspace of modular symbols defined in definition \ref{typespace}.
\begin{enumerate}
\item  The space $X_f^+$ is preserved by the action $\rho$ of $\tilde{S}(K)$ on $\bS_k(\Gamma_H(Np^r))$.
\item  In fact, there exists a unique action $\rho_f\from K\to \GL(X_f^+)$ extending $\rho$ for which
\[
\rho_f\ptbt{a}{b}{c}{d}\sigma^+_f=\omega_p(d)^{-1}\sigma^+_f,\;a,d\in \Z_p^\times,\; c\equiv 0\pmod{p^r}\\
\]
In particular, the central character of $\rho_f$ is $\omega_p^{-1}$.
\item  Let $\tau$ be the contragredient of $\rho_f$. The pair $(K,\tau)$ is a cuspidal type, and $\Ind_K^G\tau\isom \pi_{f,p}$.
\end{enumerate}
\end{theorem}

The proof of Thm.~\ref{identifytypespace} will occupy the following three subsections.  Before proving it, we record a lemma concerning the type space $X_f^+$.

\begin{lemma}
\label{propertiesofXf}
The space $X^+_f$ is the direct sum of the eigenspaces $\sigma_{f \otimes \chi}^{\chi(-1)}$ for characters $\chi$ modulo $p^{\floor{r/2}}$. In particular, it has dimension $\phi(p^{\floor{r/2}})$ and is stable under the operators $T_m$ and $\langle d \rangle$ for all $m$ and $d$, and under the Atkin-Lehner operator $W_{p^r}$.
\end{lemma}

\begin{proof}
We define operators $R_\chi$, $\Sigma$ and $m_p$ in $\End \bS_k(\Gamma_H(Np^r))$ as follows.  For a Dirichlet character $\chi$ modulo $p^{\floor{r/2}}$, define
\[ R_\chi=\sum_{u \in (\Z / p^{\floor{r/2}} \Z)^\times} \chi(u)^{-1} \tbt{1}{p^{-\floor{r/2}}u} 0 1.\]
Let $\Sigma= \tbt 1 0 0 {-1}$ be the star involution.  Finally, for $m\in \Z$ relatively prime to $Np$, let $m_p\in (\Z / N p^{\floor{r/2}} \Z)^\times$ denote the unique element congruent to $1 \pmod N$ and $m \pmod {p^{\floor{r/2}}}$.

 We gather the following elementary commutation relations:
\begin{align}
T_m \circ R_\chi &= \chi(m) R_\chi \circ T_m \label{a}\\
\Sigma \circ R_\chi &= \chi(-1) R_\chi \circ \Sigma \label{b}\\
T_m \circ W_{p^r} &= W_{p^r} \circ T_m \circ \langle m_p \rangle^{-1} \label{c}\\
\Sigma \circ W_{p^r} &= \langle (-1)_p \rangle \circ W_{p^r} \circ \Sigma \label{d}.
\end{align}

For any character $\chi$ modulo $p^{\floor{r/2}}$, we have $R_\chi(\sigma^+_f)\in X_f^+$.  Hence $X_f^+$ contains the span of the $R_\chi(\sigma^+_f)$.  In fact $X_f^+$ equals the span of the $R_\chi(\sigma_f^+)$, by Fourier inversion.   Relations \eqref{a} and \eqref{b} show that $R_\chi(\sigma^+_f)$ lies in the space $\bS_k(\Gamma_H(Np^r))[f \otimes \chi]$ and is an eigenvector for $\Sigma$ with eigenvalue $\chi(-1)$. This forces $R_\chi(\sigma^+_f)$ to be a scalar multiple of $\sigma_{f \otimes \chi}^{\chi(-1)}$, since the $\Sigma$-eigenspaces in $\bS_k(\Gamma_H(Np^r))[f \otimes \chi]$ are 1-dimensional. Note that the eigenspaces $\sigma_{f \otimes \chi}^{\chi(-1)}$ are all distinct, since by Lemma~\ref{cmtwist} we can only have $f \otimes \chi = f \otimes \chi'$ for distinct characters $\chi$ and $\chi'$ if $\chi(-1) = - \chi'(-1)$.

Moreover, by relation \eqref{c}, for each $\chi$, $W_{p^r}(\sigma_{f \otimes \chi}^{\chi(-1)})$ lies in the eigenspace $\bS_k(\Gamma_H(Np^r))[f \otimes \eps_p^{-1}\chi^{-1}]$ where $\eps_p$ is the $p$-part of the character of $f$; and \eqref{d} implies that it is a $\Sigma$-eigenvector with eigenvalue $\eps_p(-1)\chi(-1)$. Thus $X_f^+$ is closed under the action of $W_{p^r}$.
\end{proof}

\subsection{Shimura curves}
\label{Shimuracurves}

The type space $X_f^+$ admits an action of the group $\tilde{S}(K)$, which we intend to extend to an action of $K$, the maximal compact-mod-center subgroup of $\GL_2(\Q_p)$. For this we must perform a comparison between two sorts of objects. On the one hand, we have the space of modular symbols $\bS_k(\Gamma_H(Np^r))$, which is amenable to computation and which admits an action by an ``$\SL_2$ group''. On the other hand, there is the cohomology space of a (usually disconnected) Shimura curve, which has a direct interpretation in terms of automorphic representations and which admits an action by a ``$\GL_2$ group''.

We recall the setup of Thm.~\ref{identifytypespace}.  Let $p$ be prime, $r\geq 2$ and integer, and $N\geq 1$ an integer prime to $p$.
Let $H_N\subset (\Z/N\Z)^\times$ be any subgroup, let $H_p\subset (\Z/p^r\Z)^\times$ be the kernel of reduction modulo $p^{\floor{r/2}}$, and let $H=H_N\times H_p\subset (\Z/Np^r\Z)^\times$.   Finally let $\Gamma=\Gamma_H(Np^r)$.

We construct the Shimura curve analogue of the modular curve $Y(\Gamma)$.  Let $U\subset \GL_2(\hat{\Z})$ be the open subgroup consisting of matrices $\tbt{a}{b}{c}{d}$ with
\begin{enumerate}
\item $c \equiv  0\pmod{Np^r}$
\item $d\; (\text{mod }N)\in H_N$
\item $a,d\; (\text{mod }p^r) \in  H_p.$
\end{enumerate}
Thus if $U=\prod_\ell U_\ell$ is the obvious decomposition into open subgroups $U_\ell\subset\GL_2(\Z_\ell)$, then $U_p$ is as in Eq.~\eqref{UpKn}.  That is, $U_p$ is conjugate in $\GL_2(\Q_p)$ to $K_n$, where $K$ is unramified (resp., ramified) and $r=2n$ (resp., $n+1$) as $n$ is even (resp., odd).

We write $\mathcal{H}^+$ for the usual upper half plane, and $\mathcal{H}^{\pm}$ for $\C\backslash\R$.  Let $\Sh_U$ be the Shimura curve of level $U$:
\[ \Sh_U =  \GL_2(\Q)\backslash\left(\GL_2(\mathbf{A}_f)\times \mathcal{H}^{\pm}\right) / \Gamma. \]

We have $U\cap \GL_2(\Q)=\Gamma$.  Therefore the connected components of $\Sh_U$ are in bijection with $\hat{\Z}^\times/\det U = (\Z/p^{\floor{n/2}})^\times$, and each connected component is isomorphic to the modular curve $Y(\Gamma)$.

Fix a weight $k\geq 2$, and let $V=\Sym^{k-2}\C^2$ be the $(k-2)$-fold symmetric power of the tautological representation of $\GL_2(\Q)$.
We write $\mathcal{L}_k$ for the corresponding vector bundle on $\Sh_U$.  The following theorem relates the cohomology of $\mathcal{L}_k$ to automorphic representations arising from newforms of level $\Gamma$:

\begin{theorem}
\label{decompositionofH1}
The space $H^1_P(\Sh_U,\mathcal{L}_k)$ decomposes as a direct sum
\[
H^1_P(\Sh_U,\mathcal{L}_k)\isom \bigoplus_f H^1(\mathfrak{g},K_\infty,\pi_{f,\infty}\otimes V)\otimes \left(\pi_{f,\finite}\right)^U
\]
over cuspidal newforms $f$ of weight $k$. Here $\pi_f=\pi_{f,\finite}\otimes\pi_{f,\infty}$ is the automorphic representation corresponding to $f$, $\mathfrak{g}$ is the Lie algebra of $\GL_2(\R)$, and $K_\infty\subset \GL_2(\R)$ is a maximal compact subgroup. For each $f$, the Lie algebra cohomology $H^1(\mathfrak{g},K_\infty,\pi_{f,\infty}\otimes V)$ has dimension 2.
\end{theorem}

\begin{proof}
For the decomposition of $H^1_P(\Sh_U,\mathcal{L}_k)$ into subspaces indexed by automorphic representations, see~\cite{Carayol1983}, \S2.2.4. For the statement about $H^1(\mathfrak{g},K_\infty,\pi_{f,\infty}\otimes V)$, see~\cite{RogawskiTunnell}, Prop. 1.5(b).
\end{proof}

Recall the group $K'$ of Eq.~\eqref{Kprime}:  $K'$ contains $U_p$ as a normal subgroup, so there is a natural action $K'\to \Aut\Sh_U$.

\subsection{Modular symbols on \texorpdfstring{$\Sh_U$}{ShU}}  We have a decomposition of $\Sh_U$ into connected components:
\begin{equation}
\label{ShUdecomp}
\Sh_U = \coprod_a Y(\Gamma)_a
\end{equation}
where $a$ runs over a set of representatives for $\hat{\Z}^\times/\det U$ and $Y(\Gamma)_a$ is the (isomorphic) image of $Y(\Gamma)$ under the map induced by $z\mapsto \left(\tbt{a}{}{}{1},z\right)$.  The group $K'$ acts on $\Sh_U$, and the stabilizer of any component is

\[
\tilde{S}(K)' =
\set{g\in K'\biggm\vert \det g\in p^\Z(1+p^{\floor{r/2}}\Z_p)}.
\]

Let $\bS_k(U)$ be the space of cuspidal modular symbols of weight $k$ on $\Sh_U$.  Thus $\bS_k(U)$ is the direct sum of copies of $\bS_k(\Gamma)$ indexed by $a$ as in Eq.~\eqref{ShUdecomp}.  The space $\bS_k(U)$ admits an action $\rho\from K'\to\GL(\bS_k(U))$ as well as an action of the Hecke algebra.

Let $f\in S_k(\Gamma)$ be a $p$-primitive newform of level $Np^r$,
and let $\sigma_f^+\in\bS_k(\Gamma)$ be a plus eigensymbol for $f$.  Let
\[\eta_f^+=\sum_a \rho\ptbt{a}{}{}{1} \sigma_f^+ \in \bS_k(U). \]
Then the Hecke operators act on $\eta_f^+$ with the same eigenvalues as they do on $f$ (or on $\sigma_f^+$).

\begin{lemma}
\label{Wf} Let $W_f^+\subset \bS_k(U)$ be the span of the translates of $\eta_f^+$ under the $\tbt{1}{p^{-\floor{r/2}}\Z_p}{}{1}$.  Then $W_f^+$ is preserved by $K'$.  Let $\tau$ be the representation of $K$ which is contragredient to the representation of $K$ on $W_f^+$.  Then $(K',\tau)$ is a cuspidal pair and $\Ind_{K'}^G\tau\isom \pi_{f,p}$.
\end{lemma}

\begin{proof}
The integration pairing
\[ \bS_k(U)\times H^1_P(\Sh_U,\mathcal{L}_k)\to \C\]
is perfect and equivariant for the actions of both the Hecke algebra and the group $K'$.   By Prop.~\ref{decompositionofH1}, $\eta^+_f$ lies in a $K'$-submodule of $\bS_k(U)$ isomorphic to a direct sum of copies of $\left(\check{\pi}_{f,p}\right)^{U_p}$.  Furthermore, a calculation shows that
$\tbt{a}{b}{c}{d}\in\GL_2(\Z_p)$ acts on $\eta_f^+$ as $\omega_p(d)^{-1}$ whenever $c\equiv 0\pmod{p^r}$, where $\omega_p$ is the central character of $\pi_{f,p}$.  The lemma now follows from Cor.~\ref{Kmodulecor}.
\end{proof}

\subsection{Conclusion of the proof of Thm.~\ref{identifytypespace}}  Consider the projection $\Phi\from\bS_k(U)\to \bS_k(\Gamma)$ which restricts a modular symbol to the component $Y(\Gamma)_1\subset \Sh_U$.  Then $\Phi$ transforms the action of $\beta^{-1}\tilde{S}(K)\beta$ on $\bS_k(U)$ into the action of $\tilde{S}(K)$ on $\bS_k(U)$ described in \S\ref{typespacecuspidalpairs}.  We have $\Phi(\eta_f^+)=\sigma_f^+$, and consequently $\Phi(W_f^+)= X_f^+$.  On the other hand, $\dim W_f^+=\dim X_f^+$ by Prop.~\ref{propertiesofXf}.  Therefore the restriction of $\Phi$ to $W_f^+$ furnishes an isomorphism $W_f^+\to X_f^+$ compatible with the actions of $\tilde{S}(K)'$ and $\tilde{S}(K)$.  This establishes part (1) of the theorem.  The group $K'$ preserves $W_f^+$ and therefore induces an action of $K$ on $X_f^+$.  Part (2) follows from the corresponding fact concerning $\eta_f^+$, and part (3) follows from Lemma~\ref{Wf}.

\section{The algorithm}

We now use the above results to describe an explicit algorithm to determine $\pi_{f, p}$, given a new eigenform $f \in S_k(\Gamma_1(Np^r), \eps)$ with $r \ge 1$ (specified by the sequence of Hecke eigenvalues $a_\ell(f)$ for $\ell \nmid N$, as elements of some number field).

\newcounter{algstep}
\begin{list}{Step \arabic{algstep}:}{\usecounter{algstep}}
\item \label{step:PS} Check whether $a_p(f) = 0$. If $a_p(f) \ne 0$, then by Thm.~\ref{supercuspidalcriterion}, one of the following must hold:
 \begin{itemize}
  \item The $p$-part of the conductor of $\eps$ is $p^r$. In this case, $\pi_{f, p}$ is isomorphic to the principal series $\pi(\chi_1, \chi_2)$, where $\chi_1$ is the unramified character with $\chi_1(p) = a_p(f) / p^{\frac{1-k}2}$, and $\chi_2$ is the unique character such that $\chi_2 |_{\Z_p} = \eps_p^{-1}$ and $\chi_2(p) = \eps_N(p) / \chi_1(p)$, where $\eps_p$ and $\eps_N$ are respectively the $p$-part and the prime-to-$p$-part of $\eps$.
  \item $r = 1$ and the conductor of $\eps$ is prime to $p$. In this case, $a_p(f)^2 = p^{k-2} \eps(p)$, and $\pi_{f, p}$ is isomorphic to the special representation $\St \otimes \chi$, where $\chi$ is the unramified character with $\chi(p) = a_p(f) / p^{\frac{k-2}{2}}$.
 \end{itemize}
If $a_p(f) = 0$, we continue to step \ref{step:Msymb}.
\item \label{step:Msymb} Calculate a basis for the space of modular symbols $\bS_k(\Gamma_H(Np^r))$ of \S\ref{modsymbdefinitions}. Calculate the matrices representing the action on this space of the star involution $\Sigma$ and the Hecke operators $T_\ell$ for primes $\ell \ne p$ up to the Sturm bound. Use these to calculate the plus eigensymbol $\sigma^+_f$ attached to $f$, by intersecting the kernels of $\Sigma - 1$ and $T_\ell - a_\ell(f)$ until a 1-dimensional space is obtained.
\item \label{step:Xf} Calculate the subspace $X_f^+ \subseteq \bS_k(\Gamma_H(Np^r))$ spanned by the orbit of $\sigma^+_f$ under $\tbt{1}{p^{-u}\Z}{}{1}$.
\item \label{step:IsPrimitive} Calculate the matrices of the two degeneracy maps from $\bS_k(\Gamma_H(Np^r))$ to $\bS_k(\Gamma_{H'}(Np^{r-1}))$, where $H'$ is the image of $H$ modulo $Np^{r-1}$. The sum of the images of $X_f^+$ is a Hecke-invariant subspace; if it is nonzero, then it contains an eigenform, which corresponds to a twist $f'$ of $f$ having level $Np^{r-1}$. In this case, go back to step 1 with $f'$ in place of $f$. If this does not occur, we may continue to step \ref{step:SK} knowing that $f$ is $p$-primitive.
\item \label{step:SK} Compute the matrix of $\rho(g)$ acting on $X_f^+$ for each $g$ in a set of generators for the group $\tilde S(K) / S(K_n)$, where the groups $\tilde S(K)$ and $S(K_n)$ and the representation $\rho$ are as defined in \S\ref{typespacecuspidalpairs}. One may take the set of generators to be
\[ \left\{ \tbt 1 1 0 1, \tbt 0 {-1} 1 0, \tbt p 0 0 p \right\} \]
for $K$ unramified, or
\[ \left\{ \tbt 1 1 0 1, \tbt 1 0 p 1, \tbt a 0 0 {a^{-1}}, \tbt 0 {-1} p 0 \right\} \]
for $K$ ramified (where $a$ is any generator of $(\Z / p\Z)^\times$); note that $\rho\ptbt 1 1 0 1$ has been computed in the course of step \ref{step:Xf}.
\item \label{step:K0}  For each $a$ running through a set of generators of $\Z_p^\times$, determine the space of operators $\lambda_a\in \GL(X_f^+)$ satisfying $\lambda_a\rho(g)=\rho(\delta_a g\delta_a^{-1})\lambda_a$, where $g$ runs through a set of generators of $S(K)$ and where $\delta_a=\tbt{a}{}{}{1}$.  By Thm.~\ref{identifytypespace}, there will be a unique system $\lambda_a$ of such operators which each fix $\sigma_f^+$.  Let $\rho_f$ be the unique extension of $\rho$ to $K$ which satisfies $\rho_f(\delta_a)=\lambda_a$ for each $a$.
\item Output the representation $(K, \rho_f)$.  By Thm.~\ref{identifytypespace} we have $\pi_{f,p}\isom\Ind_K^{\GL_2(\Q_p)}\check{\rho}_f$.
\end{list}

\section{Examples}

\begin{example}
Let $f$ be the newform $q - q^2 + q^3 + q^4 + \dots$ of weight 2 and level 50, and take $p = 5$. Since $a_p(f) = 0$, we proceed to step \ref{step:Msymb}. We find that the subgroup $\Gamma = \Gamma_H(50)$, where $H$ consists of the elements of $(\Z / 50\Z)^\times$ that are $1 \pmod 5$, has index 360 in $\SL_2(\Z)$, and the space $\bS_2(\Gamma)$ has dimension 31. Continuing to step \ref{step:Xf}, we find that $X_f^+$ is 4-dimensional. At step \ref{step:IsPrimitive}, we find that $X_f^+$ is contained in the 18-dimensional subspace that is the intersection of the kernels of the two degeneracy maps to level 10. Thus $f$ is indeed $5$-primitive.

Since the exponent of $5$ in the conductor of $f$ is even, the group $K=\Q_p^\times\GL_2(Z_p)$ is unramified.  In step \ref{step:SK}, we compute the representation $\rho\from S(K)\to \GL(X_f^+)$.  The matrices of $\rho\ptbt{1}{1}{0}{1}$ and $\rho\ptbt{0}{-1}{1}{0}$ are respectively
\[ A_1 = \left(\begin{array}{rrrr}
1 & 0 & 0 & -1 \\
0 & 0 & -1 & 0 \\
-1 & 1 & -2 & 1 \\
1 & 0 & -1 & 0
\end{array}\right)
\quad \text{and} \quad
A_2 = \left(\begin{array}{rrrr}
-1 & 1 & -2 & 1 \\
0 & 0 & -1 & 0 \\
0 & -1 & 0 & 0 \\
0 & -1 & -1 & 1
\end{array}\right)\]
with respect to a certain basis of $X_f^+$ whose first vector is $\sigma_f^+$.

Since $f$ has trivial character, $\rho\ptbt{5}{0}{0}{5}$ is the identity.  In fact $\rho$ is already irreducible as a representation of $S(K)$, as
there are no non-scalar $4 \times 4$ matrices commuting with both $A_1$ and $A_2$, so $X_f^+$ is irreducible as a representation of $\SL_2(\Z / 5\Z)$.

Continuing to step \ref{step:K0}, we choose $a = 2$ as a generator of $(\Z / 5\Z)^\times$, and compute the action of the conjugated generators $\tbt{1}{a^{-1}}{0}{1}$ and $\tbt{0}{-a^{-1}}{a}{0}$. Either by a direct modular symbol computation, or (more efficiently) by expressing the conjugated generators as words in the original generators and using the matrices already computed, we find that these two elements act by
\[A_1' = \left(\begin{array}{rrrr}
-2 & 1 & -1 & 1 \\
-2 & 2 & -2 & 1 \\
-1 & 2 & -1 & 0 \\
-2 & 2 & -1 & 0
\end{array}\right)\text{and} \quad
A_2' = \left(\begin{array}{rrrr}
-1 & 2 & -1 & 0 \\
0 & 1 & 0 & 0 \\
0 & 0 & 1 & 0 \\
0 & 1 & 1 & -1
\end{array}\right).\]
A straightforward computation shows that the space of matrices such that $x A_i = A_i' x$ for $i = 1, 2$ is 1-dimensional, spanned by the matrix
\[B = \left(\begin{array}{rrrr}
1 & -2 & 1 & 0 \\
0 & -1 & 0 & 1 \\
0 & 0 & -1 & 1 \\
0 & -1 & -1 & 1
\end{array}\right).\]
As mentioned above, the first vector in our basis of $X_f^+$ is $\sigma^+_f$. So $B$ acts trivially on $\sigma^+_f$, and the desired extension of $\rho$ to a representation of $\GL_2(\Z / 5\Z)$ is given by $\rho\ptbt{2}{0}{0}{1} = B$. Since $f$ has trivial character, $\rho$ factors through $\operatorname{PGL}_2(\Z / 5\Z)$, and we see that the character of $\rho$ is given by:
\[
 \begin{array}{|ccccccc|}
\hline
\tbt 1 0 0 1 &
\tbt 2 0 0 1 &
\tbt 4 0 0 1 &
\tbt 1 1 0 1 &
\tbt 0 2 1 0 &
\tbt 0 1 1 2 &
\tbt 0 2 1 2 \\
\hline
4 &
0 &
0 &
-1 &
-2 &
1 &
1\\
\hline
\end{array}
\]
In particular, using the trace relation of theorem \ref{charcalc} we see that $\rho$ corresponds to the admissible pair $(\Q_{25}, \theta)$, where $\Q_{25}$ denotes the unramified quadratic extension of $\Q_5$ and $\theta$ is either of the two characters of $\Q_{25}^\times$ trivial on $\Q_5^\times$ and having order 3. (Since these two characters are interchanged by the nontrivial element of $\Gal(\Q_{25} / \Q_5)$, they correspond to equivalent admissible pairs.)
\end{example}

\begin{example}
 Let $\lambda$ be any root of $x^4 + 9 = 0$. Then $\lambda^2 / 3$ is a square root of $-1$; we denote this by $i$. Then there is an eigenform $f \in S_3(\Gamma_1(25))$ with $q$-expansion
\[ f = q + \lambda q^2 + i \lambda q^3 - i q^4 - 3 q^6 - 4\lambda q^7 + \dots\]
whose character is the unique character $\chi$ modulo 5 mapping 2 to $i$. Then, since $\chi$ generates the group of characters modulo 5, a basis for $X_f^+$ is given by the four eigensymbols $\sigma_f^+, \sigma_{f \otimes \chi}^-, \sigma_{f \otimes \chi^2}^+$ and $\sigma_{f \otimes \chi^3}^-$. All of the corresponding forms are are in fact Galois-conjugate to $f$ over $\Q$ (they are inner twists of $f$).

One readily calculates that $X_f^+$ is irreducible as an $\SL_2(\Z/5\Z)$-representation, and the space of matrices intertwining $X_f^+$ with its conjugate by $\tbt 2 {}{} 1$ is 1-dimensional; any such matrix has the four eigensymbols above as its eigenvectors, with eigenvalues $\pm t, \pm it$ for some scalar $t$. Thus, in particular, there is a unique extension of $\rho$ to a representation of $\GL_2(\Z / 5\Z)$ with $\rho\ptbt 2 {}{} 1 \sigma^+_f = \sigma^+_f$. (Note that $g = f \otimes \chi^2$ has the same character as $f$, and moreover $X_f^+$ and $X_g^+$ are the same space. The extensions to $\GL_2(\F_5)$ corresponding to $f$ and $g$ differ only in the choice of sign of $\rho \ptbt 2 {}{} 1$, and in particular they coincide as representations of the index 2 subgroup of $\GL_2(\F_5)$.)

One checks that $\F_{25}^\times$ is generated by a root $\alpha$ of $x^2 - x + 2$, so the non-split torus in $\GL_2(\Z / 5\Z)$ is generated by $g = \tbt{1}{-1}{2}{0}$. If $\theta, \theta^\sigma$ are the characters of $\Q_{25}^\times$ corresponding to $\rho$, then we find that $\tr \rho(g) = -\lambda$, so $\theta(\alpha) + \theta(\alpha^\sigma) = \lambda$. Since $\theta(\alpha) \theta(\alpha^\sigma) = \theta(\alpha \alpha^\sigma) = \theta(2) = i$, we see that $\theta(\alpha)$ and $\theta(\alpha^\sigma)$ are the roots of $x^2 - \lambda x + i$, which are roots of unity of order 24.
\end{example}

\begin{rmk}
If $\theta$ is a character of $\Q_{25}^\times$ of level 1 such that $(\Q_{25}, \theta)$ is an admissible pair, then $\theta$ is determined up to unramified twists by the image of a root $\alpha$ of $x^2 - x + 2$. This must be a root of unity whose order divides 24, and if $\theta \ne \theta^\sigma$, it cannot be a factor of 4. Hence it is one of $\{3, 6, 8, 12, 24\}$. Conversely, the order of $\theta(\alpha)$ determines $\theta$ uniquely up to Galois conjugacy and unramified twists. Table \ref{table5adic} gives examples of forms corresponding to each of these equivalence classes.

\begin{table}[ht]
\begin{tabular}{|r|r|r|l|}
 \hline
 Order of $\theta(\alpha)$ & Weight & Level & $q$-expansion\\
 \hline
 3 	& 2 	& 50 & $q - q^2 + q^3 + q^4 + \dots$\\
 6	& 4	& 25 & $q + q^2 + 7q^3 - 7q^4 + \dots$\\
 8	& 3	& 50 & $q + (i - 1) q^2 + (3i + 3) q^3 - 2i q^4 + \dots$\\
 12	& 4	& 25 & $q + i q^2 - 7i q^3 + 7 q^4 + \dots$\\
 24	& 3	& 25 & $q + \lambda q^2 + i \lambda q^3 - i q^4 + \dots$ (where $\lambda^2 = 3i$)\\
 \hline
\end{tabular}
\caption{Examples of forms realising each equivalence class of supercuspidal representations of $\GL_2(\Q_5)$ of conductor $5^2$, modulo unramified twists and Galois conjugacy}
\label{table5adic}
\end{table}
\end{rmk}

\begin{example}
For the unique Galois orbit of newforms of weight 2, trivial character and level 81, of which a representative is $q + \sqrt{3}q^2 + q^4 - \sqrt{3} q^5 + \dots$, the space $X_f^+$ is a 6-dimensional representation of $\operatorname{PGL}_2(\Z / 9\Z)$, with $\tbt 3 0 0 3$ acting trivially.

One checks that the group
\[ \frac{(\Z_9 / 9 \Z_9)^\times}{(\Z_3 / 9 \Z_3)^\times}\]
is cyclic of order 12 and is generated by a root $\alpha$ of $x^2 - x + 2$. This corresponds to the generator $g = \tbt 1 {-1} 2 0$ of the non-split torus in $\operatorname{PGL}_2(\Z / 9\Z)$. We find that $\rho(g)$ has trace $-\sqrt3$, so by theorem \ref{charcalc}, $\rho$ corresponds to the pair of characters of $\Q_9^\times$ mapping $\alpha$ to the roots of the polynomial $x^2 + \sqrt 3 x + 1$.
\end{example}

\begin{example}
 Let $f$ be the newform of weight 2, trivial character and level 27. Then $X_f^+$ is a 2-dimensional representation of the normaliser of the Iwahori subgroup of $\GL_2(\Z_3)$, with trivial central character. One checks that we have $\tr \rho(g) = 0$ whenever $\det g$ does not lie in the index 2 subgroup of $\Q_3^\times / \Q_3^{\times 2}$ generated by $3$. Hence $\rho \cong \rho \otimes \chi$, where $\chi$ is the character of $\Q_3^\times$ corresponding via class field theory to the extension $E = \Q_3(\sqrt{-3})$, and we deduce that $\rho$ corresponds to the admissible pair $(E, \theta)$ for some character $\theta$ of $E^\times$ of level 2, trivial on $\Q_3^\times$. Using theorem \ref{charcalc} one checks that $\theta(\sqrt{-3}) = -1$ and $\theta(1 + \sqrt{-3})$ is a root of unity of order 3. (In fact $f$ has CM by the field $\Q(\sqrt{-3})$, so one can obtain the same result by global methods.)

Similarly, one checks that both of the forms of weight 2, trivial character and level 54 correspond to characters of $\Q_3(\sqrt{3})$.
\end{example}

\begin{example}
In our final example, we show how our algorithm can give a complete description of the 2-adic Weil-Deligne representation $\sigma_{f,2}=\sigma(\pi_{f,2})\from W_{\Q_2}\to\GL_2(\C)$ attached to a cusp form $f$ when the local component $\pi_{f,2}$ is supercuspidal but does not arise from an admissible pair.  This means exactly that the image of $\sigma_{f,2}(W_{\Q_2})$ in $\PGL_2(\C)$ is not a dihedral group:  it must instead be isomorphic to $A_4$ or $S_4$.  Weil showed in~\cite{WeilExercicesDyadiques} that the $A_4$ (tetrahedral) case cannot occur, and that the $S_4$ (octahedral) cases all arise from the 3-torsion in a handful of elliptic curves defined over $\Q_2$.

Our goal here is to calculate $\sigma_{f,2}$ for those representations of octahedral type.   Of course, $\sigma_{f,2}$ is determined by $\pi_{f,2}$ in light of Carayol's theorem.   In practice, however, it seems very difficult to pass from $\pi_{f,2}$ to $\sigma_{f,2}$ in an explicit manner.  It is much easier to find a form $g$ arising from one of Weil's elliptic curves for which $\pi_{f,2}$ and $\pi_{g,2}$ differ by an (explicit) character $\chi$ of $\Q_2^\times$.  Since $g$ arises from an elliptic curve, $\sigma_{g,2}$ may be computed directly.  Then $\sigma_{f,2}$ can be calculated because it differs from $\sigma_{g,2}$ by the same character $\chi$.

For example, let $f=q-4q^3 - 2q^5 + 24q^7 +\dots$ be the unique newform of level 8, weight 4 and trivial character.  Our algorithm shows that the type space $X_f^+$ is the 1-dimensional space spanned by $\sigma_f^+$, and this admits an action of the ramified maximal compact-mod-center subgroup $K\subset\GL_2(\Q_p)$ through a character $\rho_f$ of $K$.  The restriction of $\rho_f$ to the Iwahori subgroup is given by $\rho_f\ptbt a b {2c} d =(-1)^{b + c}$.  As for the uniformizer $\Pi = \tbt 0 {1} {-2} 0$,  we have $\rho_f(\Pi)=1$.    Since $\rho_f$ is 1-dimensional, there are no nontrivial characters $\chi$ such that $\rho \cong \rho \otimes \chi$; thus $\pi_{f, 2}$ is an ``exceptional'' representation of $\GL_2(\Q_2)$ in the sense of \cite[\S 44.1]{BushnellHenniartLocalLanglands}.  This means that the irreducible Weil-Deligne representation $\sigma_{f,2}\from W_{\Q_2}\to\GL_2(\C)$ has projective image $S_4$.

Now let $g = q - q^3 - 2q^5 + \dots$ be the newform corresponding to the unique isogeny class of elliptic curves over $\Q$ of conductor 24.  One such elliptic curve has Weierstrass equation $y^2 = x^3 - x^2 - 4x + 4$.  We find that $\rho_g$ is also one-dimensional;  the character $\rho_g$ differs from $\rho_f$ only in that $\rho_g(\Pi)=-1$.  Thus $\pi_{f, 2} = \pi_{g, 2} \otimes \chi$, where $\chi$ is the unramified quadratic character of $\Q_2^\times$.  This reduces the calculation of $\rho_{f,2}$ to that of $\rho_{g,2}$.  Let $K=\Q_2(E[3])$.   Then $\Gal(K/\Q_2)\isom \GL_2(\mathbf{F}_3)\isom \tilde{S}_4$ is the nonsplit double cover of $S_4$.  The $\tilde{S}_4$-extensions of $\Q_2$ are classified in~\cite{WeilExercicesDyadiques}, \S 36.  The exponent of 2 in the conductor of $E$ is 3, and this is already enough to determine $K$ completely:  By the calculations appearing in \S 6 of~\cite{RioDyadicExercises}, $K/\Q_2$ is the unique $\tilde{S}_4$-extension whose discriminant has $2$-adic valuation 64.

Since $\rho_{g,2}$ is an irreducible representation of the Weil-Deligne group, we may write $\sigma_{g,2}=\sigma\otimes\psi$, where $\psi$ is an unramified character and $\sigma$ is of ``Galois type" (meaning it has finite image, or equivalently that it extends to a continuous representation of $\Gal(\overline{\Q}_2/\Q_2)$), see~\cite{TateNumberTheoreticBackground}, (2.2.1).  Here we may take $\sigma$ to be a representation of $W_{\Q_2}$ which factors through a faithful representation $\Gal(K/\Q_2)\isom\GL_2(\mathbf{F}_3)\injects\GL_2(\C)$.  Then $\sigma$ is unique up to isomorphism.  The character $\psi$ sends a Frobenius element $\text{Fr}$ to $\sqrt{-2}$:  this is forced by the relation $\det\sigma_{g,2}({\text{Fr}})=2$.  These observations determine $\sigma_{g,2}$ completely;  hence they determine $\sigma_{f,2}$ completely as well.
\end{example}

\bibliographystyle{amsalpha}
\bibliography{ComputationOfLocalComponents}

\end{document}